\documentclass[a4paper,11pt]{article}

\usepackage{amsmath,amssymb,bm,ulem,amsthm}
\usepackage[dvipdfmx]{graphicx}
\usepackage{here}
\usepackage{hyperref}

\usepackage[dvipdfmx]{graphicx}

\theoremstyle{plain}
\newtheorem{lemma}{Lemma}
\newtheorem{prop}[lemma]{Proposition}
\newtheorem{thm}[lemma]{Theorem}
\newtheorem{corollary}[lemma]{Corollary}

\title{Some calculations of centralizer rings of a complex reflection group}

\author{ Masashi Kosuda, Manabu Oura, and Sarbaini\footnote{Corresponding author}
}

\begin{document}

\maketitle

\textbf{Abstract}.
Let $H_1$ be the complex reflection group of order $96$.
For the tensor products of faithful transitive permutation representations of $H_1$,
we determine the structures of the centralizer rings.
This complements the work of Imamura-Kosuda-Oura.

\section{Introduction}
This is a continuation of \cite{kosuda-oura,imamuraetal}.
In \cite{imamuraetal}, we took up a  transitive permutation representation of $H_1$.
The purpose of this paper is to extend the result in \cite{imamuraetal}
to all faithful transitive permutation representations of $H_1$.

The representation
theory of $H_1$ is investigated and the centralizer rings of the tensor representation of $H_1$ are determined
in \cite{kosuda-oura}. Since the character table of  $H_1$  is used in this study, it is reproduced at the end of this paper from
\cite{kosuda-oura}.

As usual, let $\mathbb{C}$ denote the complex number field. We denote by $M_d$ the matrix algebra of degree $d$
over $\mathbb{C}$. For simpliticy, let $nM_d$ denote $\underbrace{M_d\oplus\cdots\oplus M_d}_n$.

The calculations in this paper were done with the help of Magma \cite{magma}, Maple and Sagemath \cite{sage}.
\section{Preliminaries}
We recall the theory of permutation groups and representation theory of a finite groups.

 Let $G$ be a finite group and $\Omega$ be a finite set.
If $G\times \Omega \rightarrow \Omega \ ((a,\alpha)\mapsto \alpha^a)$
satisfies
\begin{align*}
\alpha^{ab}&=(\alpha^a)^b,\\
\alpha^1&=\alpha,
\end{align*}
we say that $G$ acts on $\Omega$.
This induces a group homomorphism
\begin{align*}
G & \rightarrow S^{\Omega}\\
a & \mapsto \begin{pmatrix} \alpha \\ \alpha^a
\end{pmatrix}_{\alpha\in \Omega}
\end{align*}
where $S^{\Omega}$ denotes the symmetric group on $\Omega$. The permutation character of a transitive permutation group contains the
identity character exactly once.
The centralizer ring of the transitive permutation group is commutative
if and only if the permutation character
is decomposed into the sum of irreducible characters with coefficients
$0$ or $1$ (i.e., multiplicity-free).
The transitive permutation group is doubly transitive if the permutation
character
is decomposed into the multiplicity-free sum of the two irreducible
characters,
the identity character and another non-identity character.

In general, a group homomorphism $\varphi :G \rightarrow S^{\Omega}$ is 
called a permutation representation of $G$ on $\Omega$, denoted by $(\varphi,\Omega)$.
Here, we remark that a permutation representation is a (matrix) representation in the usual sense.
Indeed, let $\Omega =\{s_1,\dots,s_n\}$.
For $\sigma\in S^{\Omega}$, we set
\[
A(\sigma)=(\alpha_{ij})
\]
where  
\[
\alpha_{ij}=\begin{cases} 1 & \text{ if }s_i^{\sigma}=s_j,\\
0 & \text{ otherwise }.
\end{cases}
\]
Then a correspondence $\sigma \mapsto A(\sigma)$ is an injective homomorphism from $S^{\Omega}$ into $GL(n,\mathbf{C})$.
Therefore a permutation representation from $G$ to $S^{\Omega}$ 
may be regarded as a group homomorphism from $G$ to $GL(n,\mathbf{C})$.

We assume that we have a permutation representation $(\varphi,\Omega)$ of $G$.
We denote by $\theta$ a permutation character of $\varphi$.
Then $\theta(a)\ (a\in G)$ counts the elements $\alpha\in \Omega$
such that
\[
\alpha^{\varphi(a)}=\alpha.
\]
If $\varphi$ is injective, it is said to be faithful.
If $\varphi(G)$ is transitive on $\Omega$, it is called a transitive representation.

Let $(\varphi,\Omega)$ and $(\varphi',\Omega')$ be permutation representations of $G$.
If there is a bijection between $\Omega$ and $\Omega'$ $(\alpha\leftrightarrow \alpha')$
such that
\[
(\alpha^{\varphi(a)})'=(\alpha')^{\varphi'(a)},
\]
we say that $(\varphi,\Omega),\ (\varphi',\Omega')$ are equivalent.
Since equivalent representations share many properties as a representation, 
we often identity them.

For a subgroup $H$ of $G$,
we set $\Omega=H\backslash G=\{Hg_1,\dots,Hg_n\}$. We let $G$ act on $\Omega$ from right
\[
\Omega\ni Hg\mapsto Hga\in \Omega.
\]
This induces a transitive permutation representation of $G$ on $\Omega$, 
\begin{align*}
\ G &\rightarrow S^{\Omega}\\
a &\mapsto \begin{pmatrix} Hg_1 & \dots & Hg_n \\ Hg_1a & \dots & Hg_na\end{pmatrix}.
\end{align*}
This representation is called a permutation representation of $G$ with respect to $H$.
Conversely, any transitive permutation representation of $G$ can be obtained in this way.
For two subgroups $H,H'$ of $G$, we set
$\Omega=H\backslash G,\ \Omega'=H'\backslash G$, respectively.
Then two permutation representations of $G$ with respect to $H$ and $H'$
are equivalent 
if and only if $H$ and $H'$ are conjugate in $G$.

By direct calculation of Magma, there are $24$ subgroups of $H_1$ upto conjugacy. Then the action of  $H_1$ on $\theta$ gives a transitive permutation representation of $H_1$. We observe that $T$ and $D$ correspond to $t$ and $d$. Consequently, we obtain a faithful permutation representation $G=\langle t,d\rangle$ of $H_1$.
Among them,  a permutation representation of $H_1$ is faithful 
with respect to the numbers 
\[
1,3,4,8,9.
\]
We denote the permutation character by $\theta_i\ (i=1,3,4,8,9)$ of each subgroup. This is \href{https://sarbaini.carrd.co/#researches}{the list of generators of subgroups and  $\langle t,d\rangle$ of $H_1$}.\cite{web}

\begin{prop}\label{pro1}
The values of the permutation characters at each conjugacy class are given as follows.

\begin{center}
	\begin{tabular}{l| c c c c c c c c c c c c c c c c c}
		
		& $\mathfrak{C}_1$&$\mathfrak{C}_2$&$\mathfrak{C}_3$&$\mathfrak{C}_4$&$\mathfrak{C}_5$&$\mathfrak{C}_6$&$\mathfrak{C}_7$&$\mathfrak{C}_8$&$\mathfrak{C}_9$&$\mathfrak{C}_{10}$&$\mathfrak{C}_{11}$&$\mathfrak{C}_{12}$&$\mathfrak{C}_{13}$&$\mathfrak{C}_{14}$&$\mathfrak{C}_{15}$&$\mathfrak{C}_{16}$\\
		\hline
		$\theta_1$ & $96$& $0$& $0$& $0$& $0$& $0$& $0$& $0$& $0$& $0$& $0$& $0$& $0$& $0$& $0$& $0$\\
		\hline
		
		$\theta_3$ & $48$& $0$& $0$& $0$& $0$& $0$& $0$& $0$& $0$& $0$& $0$& $0$& $0$& $0$& $8$& $0$\\
		\hline
		
		$\theta_4$ & $32$& $0$& $0$& $0$& $0$& $0$& $0$& $0$& $0$& $0$& $0$& $8$& $0$& $0$& $0$& $0$\\
		\hline
		$\theta_8$ & $24$& $0$& $0$& $0$& $0$& $0$& $0$& $0$& $4$& $0$& $4$& $0$& $0$& $0$& $4$& $0$\\
		\hline
		$\theta_9$ & $24$& $0$& $0$& $0$& $0$& $0$& $4$& $0$& $0$& $0$& $0$& $0$& $4$& $0$& $4$& $0$\\
	\end{tabular}
\end{center}
\end{prop}
We remark that case for $\theta_9$ is credited \cite{imamuraetal}. In order to see the differences between $\theta_8$ and $\theta_9$, we continue to include $\theta_9$ in our discussion.

\section{Results}
We follow the argument present in the papers \cite{kosuda-oura,kosuda}.
\begin{prop}\label{pro2}
	We have
	\begin{align*}
		&\theta_1=\chi_1+\chi_2+\chi_3+\chi_4+2(\chi_5+\chi_6+\chi_7+\chi_8+\chi_9+\chi_{10})\\&\quad\quad+3(\chi_{11}+\chi_{12}+\chi_{13}+\chi_{14})+4(\chi_{15}+\chi_{16}),\\
		&\theta_3=\chi_1+\chi_2+2\chi_5+\chi_7+\chi_8+\chi_9+\chi_{10}+\chi_{11}+\chi_{12}\\&\quad\quad+2(\chi_{13}+\chi_{14}+\chi_{15}+\chi_{16}),\\
		&\theta_4=\chi_1+\chi_2+\chi_3+\chi_4+\chi_{11}+\chi_{12}+\chi_{13}+\chi_{14}+2(\chi_{15}+\chi_{16}),\\
		&\theta_8=\chi_1+\chi_5+\chi_7+\chi_8+\chi_{12}+\chi_{13}+\chi_{14}+\chi_{15}+\chi_{16},\\
		&\theta_9=\chi_1+\chi_5+\chi_9+\chi_{10}+\chi_{12}+\chi_{13}+\chi_{14}+\chi_{15}+\chi_{16}.
	\end{align*}
	
\end{prop}
\begin{proof}
	Let $\theta=\theta_i\ (i=1,3,4,8,9)$ and
	\begin{align*}
		\theta=m_1\chi_1+m_2\chi_2+\cdots+m_{16}\chi_{16}.
	\end{align*}
Here we remind that $\chi_1$ is the identity character of $H_1$. Since the characters depend on the conjugacy classes, $\theta(\mathfrak{C})$ denotes the value of $\chi$ at a conjugacy class $\mathfrak{C}$. Then, we have that
\begin{align*}
	\begin{pmatrix}
		\theta(\mathfrak{C}_1)&\theta(\mathfrak{C}_2) &\dots&\theta(\mathfrak{C}_{16})
	\end{pmatrix}= \begin{pmatrix}
	m_1&m_2&\dots&m_{16}
\end{pmatrix}\textbf{X}
\end{align*}
where \textbf{X} denotes the character table of $H_1$, presented at the end of this paper. If we know the explicit $\theta(\mathfrak{C}_i)$'s on the left-hand side, we multiply \textbf{X}$^{-1}$ on both sides from the right to get the result.

Suppose $g\in \mathfrak{C}$ for a conjugacy class $\mathfrak{C}$, then $\theta(\mathfrak{C})$ is the number of $\alpha\in \Omega$ which is fixed by $g$. 
By proposition \ref{pro1},
 the values of $\theta(\mathfrak{C}_i)$'s can be given by
\begin{equation*}
\begin{pmatrix}
		\theta(\mathfrak{C}_1)&\theta(\mathfrak{C}_2) &\dots&\theta(\mathfrak{C}_{16})
\end{pmatrix}=
\begin{cases}
	(96,0,0,0,0,0,0,0,0,0,0,0,0,0,0,0),\\
	(48,0,0,0,0,0,0,0,0,0,0,0,0,0,8,0),\\
	(32,0,0,0,0,0,0,0,0,0,0,8,0,0,0,0),\\
	(24,0,0,0,0,0,0,0,4,0,4,0,0,0,4,0),\\
	(24,0,0,0,0,0,4,0,0,0,0,0,4,0,4,0).
\end{cases}
\end{equation*}
This completes the proof.
\end{proof}

We give some remarks. Corresponding permutation representations respect to $\theta_8,\theta_9$ are multiplicity free and not doubly transitive. Those centralizer rings are commutative. Because of 
\begin{align*}
\sum_{i=1}^{16} \chi_i(\mathfrak{C}_0)=36,
\end{align*} we also see the permutation representations corresponding $\theta_1$, $\theta_3$ are not multiplicity free.

 We proceed to the centralizer ring of the tensor representation of $H_1$. Let $\mathfrak{A}^{(k)}$ be the centralizer ring of the $k$-th tensor representation of $H_1$. Notice $\mathfrak{A}=\mathfrak{A}^{(1)}$. We set
\begin{align*}
	\chi^{\otimes k}=\sum_{i=1}^{16} d_i^{(k)}\chi_i,
\end{align*}
and 
\begin{align*}
	\overrightarrow{d^{(1)}}=\left(d_1^{(k)},d_2^{(k)},\dots,d_{16}^{(k)}\right).
\end{align*}
We are going to determine the coefficients $d_i^{(k)}$ explicitly. By Proposition \ref{pro2}, we already knew 

\begin{equation*}
\overrightarrow{d^{(1)}}=
\begin{cases}
	(1, 1, 1, 1, 2, 2, 2, 2, 2, 2, 3, 3, 3, 3, 4, 4),\\
	(1, 1, 0, 0, 2, 0, 1, 1, 1, 1, 1, 1, 2, 2, 2, 2),\\
	(1, 1, 1, 1, 0, 0, 0, 0, 0, 0, 1, 1, 1, 1, 2, 2),\\
	(1, 0, 0, 0, 1, 0, 1,1 ,0 ,0 ,0, 1, 1 ,1 ,1, 1),\\
	(1, 0, 0, 0 ,1, 0 ,0 ,0 ,1, 1 ,0 ,1, 1, 1, 1, 1).
\end{cases}
\end{equation*}
We shall consider the matrix $A$ such that 
\begin{align*}
	\begin{pmatrix}
		\theta\chi_1\\
		\theta\chi_2\\
		\vdots\\
		\theta\chi_{16}
	\end{pmatrix}=A
\begin{pmatrix}
	\chi_1\\
	\chi_2\\
	\vdots\\
	\chi_{16}
\end{pmatrix}.
\end{align*}
The matrix $A$ can be calculated explicitly as
\begin{align*}
	A&=\textbf{X}.\text{diag}\left(\left[\theta(\mathfrak{C}_1),\theta(\mathfrak{C}_2),\dots,\theta(\mathfrak{C}_{16})\right]\right)\textbf{X}^{-1}.
\end{align*}
 We have thus gotten 
\begin{align*}
	\overrightarrow{d^{(k)}}=\overrightarrow{d^{(k-1)}}A\;\;\; (k\geq 2).
\end{align*}
 We continue calculation to
\begin{align*}
	\overrightarrow{d^{(k)}}&=\overrightarrow {d^{(k-1)}}A\\
	&=\overrightarrow {d^{(1)}}A^{k-1} 
	\\
	&=\overrightarrow {d^{(1)}} \textbf{X}\big(\text{diag}\left(\left[\theta(\mathfrak{C}_1),\theta(\mathfrak{C}_2),\dots,\theta(\mathfrak{C}_{16})\right]\right)\big)^{k-1}\textbf{X}^{-1}\\
	&=\overrightarrow {d^{(1)}} \textbf{X}\text{diag}\left(\left[\theta(\mathfrak{C}_1)^{k-1},\theta(\mathfrak{C}_2)^{k-1},\dots,\theta(\mathfrak{C}_{16})^{k-1}\right]\right)\textbf{X}^{-1}\\
	&=(a_k,b_k,c_k,d_k,e_k,f_k,g_k,h_k,i_k,j_k,l_k,m_k,n_k,o_k,p_k,q_k).
\end{align*}
For $\theta_1$, we have
\begin{align*}
	\overrightarrow{d^{(k)}}&=(a_k,a_k,a_k,a_k,e_k,e_k,e_k,e_k,e_k,e_k,l_k,l_k,l_k,l_k,p_k,p_k),
\end{align*}
where
\begin{align*}
	a_k&=96^{k - 1},\\
	e_k&=\frac{96^k}{48},\\
	l_k&=\dfrac{96^k}{32},\\
	p_k&=\dfrac{96}{24}.
\end{align*}
For $\theta_3$, we get
\begin{align*}
	\overrightarrow{d^{(k)}}&=(a_k,a_k,c_k,c_k,e_k,f_k,g_k,g_k,g_k,g_k,l_k,l_k,n_k,n_k,p_k,p_k),
\end{align*}
where 
\begin{align*}
	a_k&=\dfrac{48^{k - 1}}{2} + \dfrac{8^{k - 1}}{2},\\
	c_k&=\dfrac{48^{k - 1}}{2} -\dfrac{8^{k - 1}}{2},\\
	e_k&=48^{k - 1} + 8^{k - 1},\\
	f_k&=48^{k - 1} - 8^{k - 1},\\
	g_k&=48^{k - 1}, \\
	l_k&=\dfrac{48^k }{32}- \dfrac{8^{k - 1}}{2},\\
	n_k&=\dfrac{48^k }{32}+ \dfrac{8^{k - 1}}{2},\\
	p_k&=\dfrac{48^k }{24}.
\end{align*}
For $\theta_4$, we get
\begin{align*}
	\overrightarrow{d^{(k)}}&=(a_k,a_k,a_k,a_k,e_k,e_k,e_k,e_k,e_k,e_k,l_k,l_k,l_k,l_k,p_k,p_k),
\end{align*}
where 
\begin{align*}
	a_k&=\dfrac{32^{k - 1}}{3} + \dfrac{8^{k}}{12},\\
	e_k&=\dfrac{32^{k}}{48} -\dfrac{8^{k}}{12},\\
	l_k&=32^{k - 1}, \\
	p_k&=\dfrac{32^{k}}{24} +\dfrac{8^{k}}{12}.
\end{align*}
For $\theta_8$, we get
\begin{align*}
	\overrightarrow{d^{(k)}}&=(a_k,b_k,b_k,b_k,e_k,f_k,e_k,e_k,f_k,f_k,l_k,m_k,m_k,m_k,p_k,p_k),
\end{align*}
where 
\begin{align*}
	a_k&=\dfrac{24^{k - 1} }{4} + 3\left( 4^{k - 2}\right),\\
	b_k&=\dfrac{24^{k - 1} }{4} - 4^{k - 2},\\
	e_k&=\dfrac{24^{k - 1}}{2} +\dfrac{4^{k - 1}}{2},\\
	f_k&=\dfrac{24^{k - 1}}{2} -\dfrac{4^{k - 1}}{2},\\
	l_k&=\dfrac{24^k}{32}-3\left( 4^{k - 2}\right), \\
	m_k&=\dfrac{24^k}{32}+ 4^{k - 2},\\
	p_k&=24^{k-1}.
\end{align*}
For $\theta_9$, we get
\begin{align*}
	\overrightarrow{d^{(k)}}&=(a_k,b_k,b_k,b_k,e_k,f_k,f_k,f_k,e_k,e_k,l_k,m_k,m_k,m_k,p_k,p_k),
\end{align*}
where  
\begin{align*}
	a_k&=\dfrac{24^{k - 1} }{4} + 3\left( 4^{k - 2}\right),\\
	b_k&=\dfrac{24^{k - 1} }{4} -4^{k - 2},\\
	e_k&=\dfrac{24^{k - 1}}{2} +\dfrac{4^{k - 1}}{2},\\
	f_k&=\dfrac{24^{k - 1}}{2} -\dfrac{4^{k - 1}}{2},\\
	l_k&=\dfrac{24^k}{32}-3\left( 4^{k - 2}\right) ,\\
	m_k&=\dfrac{24^k}{32}+ 4^{k - 2},\\
	p_k&=24^{(k-1)}.
\end{align*}
\begin{thm}
We have that
\begin{align*}
	\mathfrak{A}_{\theta_1}^{(k)}\cong
	4M_{a_{k}}\oplus 6M_{e_{k}}\oplus 4M_{l_{k}}\oplus 2M_{p_{k}}\qquad\qquad\qquad\qquad\qquad\qquad\qquad\qquad\quad k\geq 1,
\end{align*}
\begin{align*}
	\mathfrak{A}_{\theta_3}^{(k)}\cong
2M_{a_{k}}\oplus 2M_{c_{k}}\oplus M_{e_{k}}\oplus M_{f_{k}}\oplus 4M_{g_{k}}\oplus 2M_{l_{k}}\oplus 2M_{n_{k}}\oplus 2M_{p_{k}}\qquad\qquad k\geq 1,
\end{align*}
\begin{align*}
	\mathfrak{A}_{\theta_4}^{(k)}\cong
	4M_{a_{k}}\oplus 6M_{e_{k}}\oplus 4M_{l_{k}}\oplus 2M_{p_{k}}\qquad\qquad\qquad\qquad\qquad\qquad\qquad\qquad\quad k\geq 1,
\end{align*}
\begin{align*}
	\mathfrak{A}_{\theta_8}^{(k)}\cong
	M_{a_{k}}\oplus 3M_{b_{k}}\oplus M_{e_{k}}\oplus 3M_{f_{k}}\oplus 2M_{e_{k}}\oplus M_{l_{k}}\oplus 3M_{m_{k}}\oplus 2M_{p_{k}}\qquad\qquad k\geq 1,
\end{align*}
\begin{align*}
	\mathfrak{A}_{\theta_9}^{(k)}\cong
	M_{a_{k}}\oplus 3M_{b_{k}}\oplus M_{e_{k}}\oplus 3M_{f_{k}}\oplus 2M_{e_{k}}\oplus M_{l_{k}}\oplus 3M_{m_{k}}\oplus 2M_{p_{k}}\qquad\qquad k\geq 1,
\end{align*}
where $a_k,b_k,\dots,p_k$ are given above based on each its $\theta_i$
\end{thm}
\begin{corollary}\label{cor1}
	We have
	\begin{align*}
		\dim \mathfrak{A}_{\theta_1}^{(k)}&=96^{2k - 1},\\
		\dim \mathfrak{A}_{\theta_3}^{(k)}&=\dfrac{48^{2k - 1} }{2}+ \dfrac{8^{2k - 1}}{2},\\
		\dim\mathfrak{A}_{\theta_4}^{(k)}&=\dfrac{32^{2k - 1}}{3}+ \dfrac{8^{2k}}{12},\\
		\dim\mathfrak{A}_{\theta_8}^{(k)}=\dim\mathfrak{A}_{\theta_9}^{(k)}&=\dfrac{24^{2k - 1}}{4}+ 3\cdot 4^{2k-2}.
	\end{align*}
\end{corollary}
The second assertion of Corollary \ref{cor1} is obtained by taking the square sum of the
dimensions of the simple components. We conclude this paper with a small
table of dim $\mathfrak{A}^{(k)}$.\\

\begin{tabular}{r| c c c c  }
	
	&$1$&$2$&$3$&$4$\\
	\hline
	$\dim \mathfrak{A}_{\theta_1}^{(k)}$& 96& 884736& 8153726976 &75144747810816	\\
	$\dim \mathfrak{A}_{\theta_3}^{(k)}$& 28& 55552& 127418368& 293535219712\\
	$\dim \mathfrak{A}_{\theta_4}^{(k)}$& 16& 11264& 11206656& 11454644224\\
$\dim\mathfrak{A}_{\theta_8}^{(k)}=\dim\mathfrak{A}_{\theta_9}^{(k)}$& 9& 3504& 1991424& 1146630144\\
	\hline
\end{tabular}
\bigskip

\textbf{Acknowledgements}.
This work was supported by JSPS KAKENHI Grant Numbers 24K06827 and
24K06644.
The third named author of this work was supported in part by Ministry of Religious Affairs (BIB) and the Indonesia Endowment Fund for Education
(LPDP) of the Ministry of Finance of the Republic of Indonesia.

Faculty of Engineering, University of Yamanashi, 400-8511, Japan

\textit{Email address}: mkosuda@yamanashi.ac.jp

\bigskip

Faculty of Mathematics and Physics,
Institute of Science and Engineering,
Kanazawa University,
Kakuma-machi,
Ishikawa 920-1192,
Japan

\textit{Email address}: oura@se.kanazawa-u.ac.jp

\bigskip

Department of Mathematics, Universitas Islam Negeri Sultan Syarif Kasim
Riau, Indonesia,

and

Graduate School of Natural Science and Technology, Kanazawa University,
Ishikawa, 920-1192, Japan

\textit{Email address}: sarbaini@uin-suska.ac.id
\end{document}